\documentclass[a4paper,10pt,reqno]{amsart}

\usepackage{
amssymb, amsfonts, amscd, 
mathrsfs, color, graphicx}

\usepackage[symbol]{footmisc}

\numberwithin{equation}{section}
\newtheorem{theo}{Theorem}[section]
\newtheorem{lem}[theo]{Lemma}
\newtheorem{prop}[theo]{Proposition}

\newtheorem{rem}[theo]{Remark}
\newtheorem{problem}{Problem}

\theoremstyle{definition}
\newtheorem{defi}[theo]{Definition}

\newcommand{\C}{\mathbb{C}}
\newcommand{\Z}{\mathbb{Z}}
\newcommand{\N}{\mathbb{N}}

\newcommand{\M}{\mathcal{M}_N}

\newcommand{\F}{\mathcal{F}}
\newcommand{\Hil}{\mathcal{H}}
\newcommand{\Ww}{\mathcal{W}}
\newcommand{\Vv}{\mathcal{V}}

\title{Periodic source detection in Discrete dynamical systems via space-time sampling}

\author{A. Aldroubi, C. Cabrelli, U. Molter}

\date{\today}



\begin{document}

\allowdisplaybreaks[2]

\begin{abstract}
In this paper, we examine a discrete dynamical system defined by $x(n+1) = A x(n) + w(n)$, where $x$ takes values  in a Hilbert space $\Hil$ and $w$ is a periodic source with values in a fixed closed subspace $\Ww$ of  $\Hil$. 

Our goal is to identify conditions on some  spatial sampling system $G = \{g_j\}_{j \in J}$ of $\Hil$   that enable stable recovery of the unknown source term $w$  from space-time samples $\{\langle x(n), g_j\rangle\}_{n \geq 0, j \in J}$. We provide necessary and sufficient conditions on $G = \{g_j\}_{j \in J}$ to ensure stable recovery of any $w \in \Ww$. Additionally, we explicitly construct an operator $R$, dependent on $G$, such that $R\{\langle x(n), g_j\rangle\}_{n, j} = w$. 
	 
\end{abstract}

\thanks{Akram Aldroubi is supported in part by NSF Grant NSF/DMS-2208030. Carlos Cabrelli and Ursula Molter are partially supported by the Global Scholars award 2023, and Grants
PICT 2022-4875 (ANPCyT), 
PIP 202287/22 (CONICET), and 
UBACyT 20020220300154BA (UBA). }

\maketitle

%
%
%
\section{Introduction}

In this paper, we study the problem of recovering a periodic source driving a dynamical system from space-time samples. The setting is described by a discrete dynamical system in a Hilbert space $\mathcal{H}$ given by
\begin{equation}\label{DSPerSo}
x(n+1) = A x(n) + w(n),
\end{equation}
where $w: \mathbb{N} \rightarrow \mathcal{W}$, is such that $w(n+N) = w(n)$ for all $n$, and where $\mathcal{W}$ is a fixed closed subspace of $\mathcal{H}$. The variable $n \in \mathbb{N}$ represents discrete time, and the operator $A$ in Equation \eqref{DSPerSo} is assumed to be a bounded linear operator on $\mathcal{H}$. In some applications this operator  may correspond to a physical phenomenon modeling a space-time signal evolving in time under the action of the operator $A$ and driven by the periodic source $w$. The problem we wish to solve is to find the unknown periodic source $w$ from space-time samples of the function $x$, given by
\begin{equation}\label{STSamp}
y_{n,j} = \langle x(n), g_j \rangle, \quad n \in \mathbb{N}, j \in J,
\end{equation}
where $G = \{g_j\}_{j \in J} \subset \mathcal{H}$ is a countable set representing spatial test functions for probing $x(n)$ at time $n$. Specifically, we want to determine conditions on the set $G$ so that we can recover $w$ from the collected samples $\{y_{n,j}\}$ such that the recovery operator $R$ is stable under perturbations of $\{y_{n,j}\}$. This means that the operator $R$ should be a continuous operator under appropriate norms (to be defined later).

This problem was inspired by environmental monitoring for the detection of pollutants emanating from smoke stacks (see e.g., \cite{RDCV12}). The solution to this problem may also be useful in other applications such as epidemiology, network analysis, computational biology, and further applications.


The type of inverse problem we are studying falls into the so-called Dynamical Sampling problems, in which an unknown of interest in a dynamical system is to be determined from space-time samples of a time-evolving signal $x$. Dynamical sampling problems vary depending on which unknown of interest is to be found. When the unknown of interest is the initial distribution $x(0)$, it is called the space-time trade-off (see e.g., \cite{ACCMP17, ACMT17, AGHJKR21,  CMPP20, cabrelli2021multi,  DMM21, MMM21, ZLL17, Zlo22}). If the unknown is the operator $A$ or some parameters related to $A$, the problem is referred to as system identification \cite{AHKLLV18, AK16, CT22, Tan17}. The problem of interest in this paper falls into the category of source recovery problems (see e.g., \cite{ADM24, ADGMM23, AGK23, aldroubi2023dynamical, AGK23, AHKK23} and the references therein).

\subsection{Notation}
Throughout the paper, $\Hil$  will denote  a separable Hilbert space. $\Hil^N$ will denote the Hilbert space $\Hil\times\cdots\times\Hil$ with the natural inner product induced by $\Hil$. In particular, $\|h\|^2_{\Hil^N}=\sum\limits_{i=1}^N\|h_i\|^2_\Hil$ for $h \in \Hil^N$.

To simplify some of the notation, we  introduce the following operators: 

 Let $A:\Hil \rightarrow \Hil$ be a bounded invertible operator with operator norm $\|A\| < 1$, and an integer $N\ge1$,
we define $N$ invertible operators
\begin{equation}\label{Ts} T_s := \left(I e^{2 \pi i s/N}-A\right)^{-1}, \quad \text{for each}\ 0\leq s\leq N-1.
\end{equation}
Here $I$ denotes the Identity operator in $\Hil$.
For the dynamical system \eqref {DSPerSo} the sequence $w$ will be called the {\em source} term, and $N$  its period.

We now introduce two Banach spaces $\mathcal{B}(\ell^2, \ell^\infty)$, and $\M$ which we will need for out work.

\begin{defi} \label{Bsp} The space $\mathcal{B}(\ell^2, \ell^\infty)$ is the set of double infinite matrices 

$D = [d_{i,j}]_{ i,j \ge 1}$ such that each row $r_i$ of $D$ belongs to $\ell^2$, and $\sup\limits_{i \ge 1} \|r_i\|_{\ell^2} < \infty$. We endow $\mathcal{B}(\ell^2, \ell^\infty)$ with the norm $\|D\|_{\ell^2\to\ell^\infty} = \sup\limits_{i \ge 1} \|r_i\|_{\ell^2}$.
\end{defi}

The space $\mathcal{B}(\ell^2, \ell^\infty)$ is a Banach space and can be identified with the space of bounded linear operators from $\ell^2$ to $\ell^\infty$, endowed with the operator norm. (See \cite{ADGMM23}). We are now ready to define the subspace $\M$.

\begin{defi}\label{SpM} Let $N\ge 1$ be an integer. The space $\M$ is the set of matrices $\{D = [d_{i,j}]: i,j \ge 1\} \subset \mathcal{B}(\ell^2, \ell^\infty)$ with rows $r_i$ such that there exist $t_1,\cdots, t_N \in \ell^2$ satisfying $\lim_{k\to \infty} \|r_{Nk+s} - t_s\|_{\ell^2} = 0$. We endow $\M$ with the norm induced by $\mathcal{B}(\ell^2, \ell^\infty)$.
\end{defi}

In Appendix~\ref{SubspaceM} we prove properties of the space $\M$.



\subsection{Contributions and Organization}  The paper is organized as follows. In Section \ref{main}, we state  Theorem  \ref{MT} wich provide the necessary and sufficient conditions on the set $G=\{g_j\}_{j\in J}\subset \Hil$ for the stable recovery of the source term $w$  in \eqref{DSPerSo} from the space-time samples \eqref{STSamp}. The proof of  Theorem  \ref{MT} is presented in Section \ref{PMT}, which also provides a reconstruction operator $R$ in \eqref{algR}. The two sections in the Appendix develop tools that are used in the proof of Theorem \ref{MT}. Specifically, in Appendix \ref{fourier} we develop the Discrete Fourier Transform  for functions taking values in a general  Hilbert space, while in Appendix \ref{SubspaceM} we show that $\M$ is a Banach space, prove some of its properties and conclude that it is the natural domain of the reconstruction operator $R$.

\section{Main result} \label {main}
In this section, we provide the necessary and sufficient conditions on a set $G = \{g_j\}_{j \in J} \subset \Hil$ for solving the recovery of any $N$-periodic source term $w$ driving the system as described in Equation \eqref{DSPerSo}. Throughout the article, we will sometimes identify $N$-periodic sequences with the vector of their first $N$ elements, i.e.~we will represent an $N-$periodic sequence $a:\Z \rightarrow \Hil$ by the finite sequence 
$(a(0),...,a(N-1)) \in \Hil^N.$

We begin by  defining  the general problem:

\begin{problem}\label{prob: general}
Let $A: \Hil \to \Hil$ be a bounded linear operator with  $\|A\| < 1$. Let $\Ww$ be a closed subspace of $\Hil$. Consider the dynamical system describing the evolution of $x(n) \in \Hil$ via
\begin{equation}\label{DS}
    \begin{cases}
    x(n+1) = A x(n) + w(n), & \text{subject to } w(n) \in \Ww,\\
    y_{n,j} = \langle x(n), g_j \rangle, & \text{for } j \in J,
    \end{cases}
\end{equation}
where the initial state $x(0) = x_0$ is unknown, and $w(n+N) = w(n)$ is an unknown $N$-periodic sequence in $\Ww$.
What are the conditions on $G = \{g_j\}_{j \in J}\subset \Hil$ such that $w$ can be recovered from $\{y_{n,j}\}_{n \geq 0, j \in J}$ in a stable way? Equivalently, under what conditions on $G$ does there exist a linear bounded operator $R: \M \to \Ww^N$ such that $R(\{\langle x(n), g_j \rangle\}_{n \geq 0, j \in J}) = w$?
\end{problem}
Before we describe our results several remarks are in order.
 
\begin{rem}${}$
\begin{enumerate}

    \item Given $\Hil$, $\Ww \subset \Hil$, and $A$,  if $\{g_j\}_{j\in J}$ is a Bessel sequence in $\Hil$, each solution $x$ of  system  \eqref {DSPerSo} determines an element  $\{\langle x(n), g_j \rangle\}_{n \geq 0, j \in J}$ in $\mathcal{M}_N$, (See Definition \ref{SpM}).
     \item  Given a fixed $w\in W^N$, consider the subset of $\M$ of   all matrices  \\ $\{\langle x(n), g_j \rangle\}_{n \geq 0, j \in J}$ corresponding to  solutions $x$ of the system associated with $w$. 
 We require that the continuous operator $R$, described in Problem~\ref{prob: general}, has constant value $w$ on that set,. i.e., $R(\{\langle x(n), g_j \rangle\}_{n \geq 0, j \in J}) = w$ for all solutions $x$, independent of the unknown initial state $x(0).$
 \end{enumerate}
\end{rem}

The following theorem provides the necessary and sufficient conditions on the probing set of functions $G = \{g_j\}_{j \in J}$ for the existence of a continuous reconstruction operator $R$. One part of its proof describes a reconstruction algorithm. However, this algorithm is not unique, in general. 

Using the notation from Equation \eqref{Ts}, we present the main result.

\begin{theo} \label {MT}
Let $A:\Hil\rightarrow \Hil$ be a bounded operator with $\|A\| < 1$, N a positive  integer,  and $\Ww$ a closed subspace of $\Hil$.

Fix $A, N$ and $\Ww$.
The following are equivalent:
\begin{enumerate}
\item There exists a bounded operator $ {R}:\M\rightarrow \Ww^N$ such that for each solution $x$ of the system in $\eqref {DS}$ it holds that $ {R}(\{\langle x(n),g_j \rangle\}_{n\geq 0, j\in J})=w$. 

\item The set $\{P_{\Ww} (T^*_s g_j)\}_{j}$ is a frame of $\Ww$ for $s=0, \dots, N-1.$

\item $\{P_{\Ww_s}(g_j)\}_{j\in J}$ is a frame of $\Ww_s$, for $s=0, \dots, N-1$, where $\Ww_s = T_s (\Ww)$.
\end{enumerate}
\end{theo}


\section {Proofs} \label {PMT}
Before proving Theorem \ref {MT}, we start with some preliminaries. 

\subsection{Periodic solution related to the periodic source} \label {alg} 
 In this subsection,  we prove that there exists an unique periodic solution $x_p$ to the dynamical system in Equation \ref{DS}.
 
A sequence $w$ with values in $\Ww$ satisfying  
\begin{equation*} \label{w-defi}
w(n+N) = w(n), \quad \text {for all } n \in \N,
\end{equation*}
can be expressed as 
 \begin{equation}\label{w-per}  w(n)=\frac{1}{\sqrt{N}}\sum_{k=0}^{N-1} \widehat{w}(k) \, e^{2\pi i k n/N}.
\end{equation}

The next Lemma proves the existence and uniqueness of a periodic solution to \eqref {DS} and gives the relation between the unique periodic solution $x_p$ of \eqref {DS} in terms of the source $w$.
\begin{lem}\label{perio}
Let $\Hil$ be a separable Hilbert space, let $\Ww \subset \Hil$ be a subspace, let $A:\Hil \to \Hil$ be a bounded operator with $\|A\|<1$, and let $w: \Z \rightarrow \Ww$, be $N$-periodic. Then there exists a unique periodic solution $x_p$ to \eqref {DS}. Moreover, the Fourier coefficients $\widehat{x_p}(k) \in \Hil$ of the $N$-periodic solution to \eqref{DS}, i.e.,
\begin{equation} \label {FXp} 
x_p(n) := \frac{1}{\sqrt N}\sum_{k=0}^{N-1} \widehat{x_p}(k)\,e^{2\pi i k n/N}
\end{equation} 
 are related to the Fourier coefficients of the periodic source $\widehat{w}(k)$ in \eqref {w-per} by the identities
$$\widehat{x_p}(k)  \left(I\,e^{2\pi i k /N} - A\right) = \widehat{w}(k), \quad k=0,\cdots,N-1,$$
 where $I$ denotes the identity operator on $\Hil$. Equivalently, using the notation introduced in equation~\eqref {Ts},
 $$\widehat{x_p}(k) = T_k \widehat{w}(k),\quad k = 0, \dots, N-1.$$
\end{lem}
\begin{proof}
Assume $\{x_p(n)\} \subset \Hil$ is an $N$-periodic sequence that satisfies \eqref{DS}. Then, using \eqref {FXp}, we have 
\begin{align*}
x_p(n+1) &= \frac{1}{\sqrt N}\sum_{k=0}^{N-1} \widehat{x_p}(k)\,e^{2\pi i k (n+1)/N} \\
& = A\left(\frac{1}{\sqrt N}\sum_{k=0}^{N-1} \widehat{x_p}(k)\,e^{2\pi i k n/N}\right) + \frac{1}{\sqrt{N}}\sum_{k=0}^{N-1} \widehat{w}(k) \, e^{2\pi i k n/N}.
\end{align*}
Thus,
\begin{equation*}
\frac{1}{\sqrt N}\sum_{k=0}^{N-1} \left(\widehat{x_p}(k)\,e^{2\pi i k/N}\right)e^{2\pi i k n/N} = 
A\left(\frac{1}{\sqrt N}\sum_{k=0}^{N-1} \widehat{x_p}(k)\,e^{2\pi i k n/N}\right) + \frac{1}{\sqrt{N}}\sum_{k=0}^{N-1} \widehat{w}(k) \, e^{2\pi i k n/N}.
\end{equation*}
Finally, it follows that
\begin{equation*}
\frac{1}{\sqrt N}\sum_{k=0}^{N-1} \widehat{x_p}(k)\left( I\,e^{2\pi i k/N} -
A \right)\,e^{2\pi i k n/N}  = \frac{1}{\sqrt{N}}\sum_{k=0}^{N-1} \widehat{w}(k) \, e^{2\pi i k n/N},
\end{equation*}
which, by the unicity of the Fourier coefficients yields
\begin{equation} \label{per-coef}
 \widehat{x_p}(k)\left(I\,e^{2\pi i k/N} -
A  \right) = \widehat{w}(k), \ k= 0, \dots, N-1.
\end{equation}
Since $\|A\|<1$, $T_k$ is an invertible bounded operator on $\Hil$ for each $k=0,\cdots N-1$. Hence, the expression above also guarantees the existence and uniqueness of a periodic solution $x_p$. 
\end{proof}

\subsection{General solution of the difference equation with a periodic source}

 In this subsection, we obtain the relation between the periodic solution $x_p$ and the periodic source  term $w$, using the discrete  Fourier transform in Hilbert spaces described in Appendix~\ref{fourier}.

From the theory of difference equations (see for example \cite{KM-2002}), the solution of \eqref{DS}, for the initial condition $x(0) = x_0$ is given by
\begin{equation}\label{general}
x(n) = A^nx_0 + \sum_{k=0}^{n-1} A^{n-1-k}w(k),\ n=0, \dots
\end{equation}
Assuming that $\|A\| <1$, and $w(n)$ is $N$-periodic  we re-write \eqref{general} as
\begin{align}\label{gen-per}
x(n) & = A^nx_0 + \sum_{k=-\infty}^{n-1} A^{n-1-k}w(k) - \sum_{k=-\infty}^{-1} A^{n-1-k}w(k) \notag\\
 & = A^n\left(x_0 - \sum_{k=-\infty}^{-1} A^{-1-k}w(k)\right)  +\sum_{k=-\infty}^{n-1} A^{n-1-k}w(k)\notag\\
 & = A^n c +\sum_{k=-\infty}^n A^{n-k}w(k).
 \end{align}
 Here $c \in \Hil$ is defined by $c = x_0 -  \sum_{k=1}^{\infty} A^{k}w(-k-1)$ which converges by our assumption that $\|A\|<1$.
Let $v(n)$ be defined by 
\begin{equation}
v(n) = \sum_{k=-\infty}^{n} A^{n-k}w(k).
\end{equation} 
Then 
\begin{align}
v(n+N) &= \sum_{k=-\infty}^{n+N} A^{n+N-k}w(k)\notag\\
&= \sum_{\ell =-\infty}^{n} A^{n-\ell}w(\ell+N) = \sum_{\ell =-\infty}^{n} A^{n-\ell}w(\ell) = v(n).
\end{align}
Hence, $\{v(n)\}$ is an $N$-periodic sequence in $\Hil$. In addition, 
\begin{align}
v(n+1) & = \sum_{k=-\infty}^{n+1} A^{n+1-k}w(k) \\
& = \sum_{k=-\infty}^{n} A^{n+1-k}w(k) + w(n) \\
& = A\left(\sum_{k=-\infty}^{n} A^{n-k}w(k)\right) + w(n) \\
& = Av(n) + w(n).
\end{align}
Therefore $\{v(n)\}$ is an $N$-periodic sequence in $\Hil$ that satisfies \eqref{DS}.

By Lemma~\ref{perio},  we obtain 
\begin{align*}
v(n) = x_p(n) & = \sum_{k=0}^{N-1} \widehat{x_p}(k)\,e^{2\pi i k n/N} \\
& \sum_{k=0}^{N-1} \left(I\,e^{2\pi i k n/N} - A\right)^{-1}\widehat{w}(k) \,e^{2\pi i k n/N}.
\end{align*}

Therefore, the  solution for \eqref{DS} with $x(0) = x_0$ is
\begin{equation} \label {GS}x(n) = A^n c + x_p(n), 
\end{equation}
where $c = x_0 -  \sum_{k=1}^{\infty} A^{k}w(-k-1)$ and $x_p(n) = \sum_{k=0}^{N-1} \left(I\,e^{2\pi i k n/N} - A\right)^{-1}\widehat{w}(k) \,e^{2\pi i k n/N}$.

With the notation from \eqref{Ts}, the periodic solution to \eqref{DS}  can be expressed as
$$ x_p(n) = \sum_{k=0}^{N-1} T_k\widehat{w}(k) \,e^{2\pi i k n/N}. $$

\subsection{Sufficient conditions for the stable recovery of the source $w$}
\ 

In this section, we provide sufficient conditions on the set $G = \{g_j\}_{j \in J}$ for the stable recovery of any sequence $w$ taking values in  $\mathcal{W}$ from the space-time samples $\{\langle x(n), g_j \rangle\}_{n \geq 0, j \in J}$. In doing so, we also develop an algorithm for the recovery of $w$. It turns out that these sufficient conditions are also necessary, as we will prove in the next section. However, the algorithm we describe is not unique; there are many algorithms that can achieve stable recovery as long as the conditions on $G$ are met.
The theorem providing these sufficient condition is given below.
\begin{theo}\label{ida}
Consider the dynamical system \eqref {DS} and let $\{g_j\}_{j\in J} \
\subset \Hil$ be a Bessel sequence.  Then, the periodic source term $w$ can be recovered from the measurements $\{\langle x(n), g_j \rangle\}_{n\geq 0, j\in J}$
 provided that $P_{\Ww_s}\{g_j\}_{j\in J}$ is a frame of $\Ww_s= T_s (\Ww)$, for $s=0, \dots, N-1$ where $T_s$ is as in \eqref {Ts}.
\end{theo}
\begin{proof}

From~\eqref{GS}, the general solution of  \eqref{DS} is of the form
$$ x(n) = A^n c +x_p(n) $$ where $x_p$ is the unique $N$-periodic solution to system \eqref {DS} for $N$-periodic  source $w$. 

Set $n=Nk+r$ with $0\leq r \leq N-1$. Thus,

$$\langle x(Nk+r) , g_j \rangle = \langle A^{Nk} A^rc ,g_j\rangle + \langle x_p(r), g_j \rangle.$$
The first term on the right hand side can be bounded above by
$$ \vert \langle A^{Nk} A^rc ,g_j\rangle \vert \le \|A^{Nk} A^rc\|\|g_j\|\le \|A\|^{Nk} \|A^rc\| \|g_j\|.$$
Since $\|A\|<1$, it follows that $\lim_{k \rightarrow \infty}\langle A^{Nk} A^rc ,g_j\rangle =0$.
We conclude that 
 \begin {equation} \label{lim_meas}
\lim_{k \rightarrow \infty} \langle x(Nk+r) , g_j \rangle = \lim_{k \rightarrow \infty}  \langle A^{Nk} A^rc ,g_j\rangle +    \langle x_p(r), g_j \rangle =   \langle x_p(r), g_j \rangle.
\end{equation}

In what follows, we will show how, under the right conditions of $G=\{g_j\}$, we can obtain $w$ from  $\{\langle x_p(r), g_j \rangle\}_{j\in J}$.

For each $s=0,\cdots, N-1, \;\widehat{x_p}(s) = T_s \widehat{w}(s)$. Since 
$\widehat{w}(s)$ is a linear combination of the vectors $w(r), r=0,\dots,N-1$ belonging to $\Ww$, we
conclude  that $\widehat{x_p}(s) \in T_s (\Ww)$,  and  that $x_p(k) \in T_s (\Ww)$ (see Appendix A).

Let $\Ww_s$ be the closed subspace of $\Hil$, defined as $\Ww_s :=T_s (\Ww)$.
Assume that  $P_{\Ww_s}\{g_j\}_{j\in J}$ is a frame of $\Ww_s$ for $s=0, \dots, N-1$ (or equivalently, that $\{P_{\Ww} (T^*_sg_j)\}_{j}$ is a frame of $\Ww$ for $s=0, \dots, N-1$), and let $\{f_j^s\}_{j\in J}$  be one of its a dual frames in $\Ww_s$.
Then, by the reconstruction formula for frames,  we can recover $x_p(s)$ from 
$$x_p(s) = \sum_{j\in J} \langle x_p(s), g_j \rangle  f_j^s.$$
 The relation $\widehat{x_p}(s) = T_s \widehat{w}(s)$ is then used to recover   $\widehat{w}(s)$, $s=0,\cdots,N-1$  and hence $w(k)$, $k=0,\cdots,N-1$.
\end{proof}

Note that the reconstruction illustrated above implies for example the choice of a dual frame for $\{T^*_sg_j\}_{j}$ for $s=0, \dots, N-1$, which is, in general, not unique, showing that our reconstruction algorithm is not unique. However, 
we will show next,  that the  condition on the  set $G=\{g_j\}_{j\in J}$ in theorem \ref {ida} turns out to be necessary for the existence of a continuous reconstruction operator $R$, as described in the next section. 



\subsection{Proof of necessity}

 In the next proposition we  state that the frame condition in Theorem \ref{ida}, is necessary for the existence of a continuous linear operator $\widetilde{R}$ from  the Banach space $\M$ to $\Ww^N$ such that $\widetilde{R}(\{\langle x(n),g_j \rangle\}_{n\geq 0, j\in J})=\widehat{w}$ and hence an operator  $R:=\F_N^{-1}\circ \widetilde{R} : \M \to \Ww^N$ such that $R(\{\langle x(n),g_j \rangle\}_{n\geq 0, j\in J})={w}$.

\begin{prop}\label{neces}

Fix $A, N$ and $\Ww$. 
Assume that there exists a bounded operator $\widetilde{R}:\M\rightarrow \Ww^N$ such that for each solution $x$ of a  system in \eqref {DS}, we have  $\widetilde{R}(\{\langle x(n),g_j \rangle\}_{n\geq 0, j\in J})=\widehat{w}$, then the set $\{P_{\Ww} (T^*_s g_j)\}_{j}$ is a frame of $\Ww$ for $s=0, \dots, N-1.$
\end{prop}\begin{proof}

Let $v \in \Ww$ be an arbitrary vector and $0\leq s \leq N-1.$ Define $w_s\in \Ww^N$ by $\widehat{w_s}(s) =  v$, and $\widehat{w_s}(\ell)=0$ for $\ell \not= s$.
Consider the dynamical system \eqref {DS} with this choice of $w_s$ and the unique periodic solution $x_p$ associated with this choice. Note that, from Lemma~\ref{perio} we have that $\widehat x_p(\ell) = T_s \widehat{w_s}(\ell)$, and hence $\widehat x_p(s) = T_s v$ and $\widehat x_p(\ell) = 0$ for $\ell \not= s$.
Then we have,  
\begin{align*}
\|v\|_{\Hil^N}^2  =\|\widehat{w_s}\|_{\Hil^N}^2 &= \|\widetilde{R}(\{<x_p(n),g_j>\}\|_{\Hil}^2 \\
& \leq \|\widetilde{R}\|^2 \|\{<x_p(Nk+r),g_j>\}\|_B^2\\
&=\sup_{0\leq r \le N-1}\sup_k \sum_{j\in J} |<x_p(Nk+r),g_j>|^2 \\
&=\sup_{0\leq r \le N-1} \sum_{j\in J} |<x_p(r),g_j>|^2 \\
&=\sup_{0\leq r \le N-1} \sum_{j \in J} \left|\left\langle \left(\frac{1}{\sqrt{N}} \sum_{\ell=0}^{N-1} \widehat{x_p}(\ell) e^{2 \pi i \ell r/N}\right) ,g_j \right\rangle\right|^2\\
& = \sup_{0\leq r \le N-1}\sum_{j \in J} |\frac{1}{\sqrt{N}} \sum_{\ell=0}^{N-1} e^{2 \pi i \ell r/N}\langle T_\ell \widehat{w}(\ell) ,g_j \rangle |^2 \\
&= \frac{1}{N} \sum_{j\in J} |\langle T_s v ,g_j \rangle|^2\\
& =  \frac{1}{N} \sum_{j\in J} | \langle v ,T_s^* g_j \rangle |^2.
\end{align*}

So we have that for any $v\in \Ww$, 
$$N\|v\|_{\Hil^N}^2  \leq \sum_{j\in J} | \langle v ,T_s^* g_j \rangle |^2.$$
This says that $\{P_\Ww (T_s^*g_j)\}_{j\in J}$ satisfies the lower frame inequality for $\Ww$.
For the upper frame bound we have 
$$ \sum_{j\in J} | \langle v ,T_s^* g_j \rangle |^2=  \sum_{j\in J} | \langle T_sv ,g_j \rangle |^2 \leq c_1 \|T_s^* v\|^2 \leq c_2 \|v\|^2.$$
Here $c_1$ is the Bessel constant of $\{g_j\}$.
Since this argument can be repeated for each $0\leq s \leq N-1$ this concludes the proof of the implication.
\end{proof}

\subsection{Proof of Theorem~\ref{MT}}
\ 

We first need the following Lemma:
\begin{lem}\label{equiv2-3} Let $\Hil$ be a Hilbert space,  $\Vv \subset \Hil$ be a closed subspace  and $\{g_j\}_{j\in J}$  a Bessel sequence in  $\Hil$ and let $T:\Hil \rightarrow \Hil$ be a bounded invertible operator. Then 
the set $\{P_{\Vv} T^* g_j\}_{j}$ is a frame of $\Vv$ if and only if 
$\{P_{T\Vv}(g_j)\}_{j\in J}$ is a frame of $T\Vv$.
\end{lem}

\begin{proof}
For the sufficiency, let $m$ and $M$ be the lower and upper frame bounds for $\{P_{\Vv} T^* g_j\}_{j}$. Then for $v \in \Vv$,  
\begin{equation*}
 m\|v\|^2 \leq  \sum_{j \in J}| \langle v, P_{\Vv} T^*g_j \rangle|^2  = \sum_{j \in J} | \left\langle Tv ,g_j \right\rangle|^2 .
\end{equation*}
Given  $w\in T\Vv$, let $v=T^{-1}w$. We have
\begin{equation*}
\begin{split}
 \sum_{j \in J} |\left\langle w ,P_{TV}g_j \right\rangle|^2  =\sum_{j \in J} |\left\langle w ,g_j \right\rangle|^2=\sum_{j \in J} |\left\langle Tv ,g_j \right\rangle|^2 &=\sum_{j \in J} |\left\langle v ,P_{\Vv}T^*g_j \right\rangle|^2\\
 & \geq
 m\|v\|^2 =m \|T^{-1}\|^2\|w\|^2.
 \end{split}
 \end{equation*}
 The upper frame bound follows from the fact that $\{g_j\}$ is Bessel. The proof of the converse is the same.
 
\end{proof}

\begin{proof}[Proof of Theorem~\ref{MT}]
\ 

1) $\Longrightarrow$ 2)  Is simply Proposition~\ref{neces}  followed by the application of the Fourier matrix $F^*_N$ defined by \eqref{FM}.

2) $\Longrightarrow$ 1)
Let us start by defining the operator ${Q}: \M \rightarrow \Ww^N$. 

Let $\{f^r_j\}_j$ be a dual frame of $\{P_{\Ww}(T_r^*g_j\}_j$ in $\Ww$.
Given $Y= [y_{Nk+r,j}]\in \M$, define
\begin{equation} \label{algR}
Q(Y) := \left[\lim_{k\rightarrow \infty}{\sum_{j\in J} {y}_{Nk+r,j} f^r_j}\right]_{r=0,\dots,N-1}.
\end{equation}
By Lemma~\ref{LimConv}, ${Q}$ is a well defined operator which is bounded since $\{f^r_j\}_j$  is Bessel for $0\leq r \leq N-1:$
$$ \|{Q}(Y)\|_{\Hil^N} \leq C \|Y\|_{\ell^2\to \ell^\infty}.$$

Now if $x(n)$ is a solution of dynamical system \eqref {DS} with  $N$-periodic source $w$, then, by the proof of Theorem \ref{ida}, we know that $
{Q}( \{ \langle x(n), g_j \rangle\}_{n\geq 0, j\in J}) = x_p$, the unique periodic solution of \eqref{DS} with $N$-periodic source $w$.  

Using Lemma~\ref{perio} we obtain $\widehat w$, and applying the  inverse Fourier transform $\F^{-1}_N$ defined by \eqref{FM}, the result follows. So, our operator is 
$$  R := \F_N^{-1} U\F_N Q \quad\text{where}\quad U=diag(T_0^{-1}, \cdots, T_{N-1}^{-1}).$$

For $ 2) \Longleftrightarrow 3)$ see Lemma~\ref{equiv2-3}.

\end{proof}
\appendix
\section{Discrete Fourier transform in general Hilbert spaces} \label {fourier}
 
 In this section we prove the existence of the Fourier decomposition that we use in our main theorem.
 The Fourier transform can be defined for functions taking values in a general separable Hilbert space.
 Here we use the Fourier transform for periodic vector value sequences defined in the integers. Although this is a known folklore, we present it here for completeness.
 
Let $\Hil$ be a separable Hilbert space, $N \in \Z, \; N \geq 1$, and $\Z_N$ the cyclic group of $N$ elements.
\begin{prop}
Any function
$q: \Z_N \rightarrow \Hil$
has a  Fourier decomposition, i.e. there exist unique function $\widehat{q}:\Z_N \rightarrow \Hil$ such that
$$ q(n)=\frac{1}{\sqrt{N}}\sum_{k=0}^{N-1} \widehat{q}(k) \, e^{2\pi i k n/N}, \;n\in\Z_N $$
\end{prop}

\begin{proof}

For $k\in \Z_N$ and $v\in\Hil$  define, $L_k(v)  = \sum_{s=0}^{N-1} \langle v,q(s)\rangle\,  e^{-2\pi i s k/N}$.
Note that  $L_k:\Hil\rightarrow \C$ is a linear functional,
and is bounded since
$$|L_k(v)| \leq \sum_{s=0}^{N-1} |\langle v,q(s)\rangle \,| \leq   (\sum_{s=0}^{N-1}\|q(s)\|) \;\|v\|.$$ 

So, using Riesz theorem, for each $k \in \Z_N$ there exists $\widehat{q}(k)\in\Hil$ such that $L_k(v)=\langle v,\widehat{q}(k)\rangle \,$ for all $v \in \Hil.$

Now since $n \mapsto \langle v, q(n)\rangle$ is a scalar function it has a discrete Fourier decomposition, thus\begin{align*}
\langle v, q(n)\rangle \, = & \sum_{k=0}^{N-1} \left(\sum_{s=0}^{N-1} \langle v,q(s)\rangle \,  e^{-2\pi i s k/N} \right) \,e^{2\pi i k n/N} \\
= & \sum_{k=0}^{N-1} L_k(v) \,e^{2\pi i k n/N}
=  \sum_{k=0}^{N-1} \langle v,\widehat{q}(k)\rangle \,  e^{2\pi i k n/N},
\end{align*}
for every $v \in \Hil$.
Thus, $ q(n)=\sum_{k=0}^{N-1} \widehat{q}(k)  e^{2\pi i k n/N}. $

\end{proof}


The unitary transformation  
$$\F_N: \ell^2(\Z_N,\Hil) \rightarrow \ell^2(\Z_N,\Hil)$$

defined by  $q \longmapsto \widehat{q}$,  is called the 
{\em Fourier transform}.
 
Thus, we have the formulae:
\begin{equation}\label{w-per-A}  q(n)=\frac{1}{\sqrt{N}}\sum_{k=0}^{N-1} \widehat{q}(k) \, e^{2\pi i k n/N}, \qquad
 \widehat{q}(k)=\frac{1}{\sqrt{N}}\sum_{n=0}^{N-1} q(n) \, e^{-2\pi i n k/N}.
 \end{equation}


The Fourier transform can equivalently be defined component-wise:

Fix an orthonormal basis $\{v_k\}_{k\in I}$ of $\Hil.$
Given $q: \Z_N \rightarrow \Hil$  define for each $k\in I,\; q_k:\Z_N \rightarrow \C\;$ by
$q_k(j) =\langle q(j), v_k \rangle, \;\forall j\in \Z_N.$

Then, we have $q(j) =\sum_{k\in I} q_k(j) \,v_k,\; j\in \Z_N$.
Now, each $q_k$ has a scalar discrete Fourier decomposition.
i.e. $q_k(j) = \frac{1}{\sqrt{N}} \sum_{s=0}^{N-1} \widehat{q_k}(s) \; e^{2\pi i j s/N}$.

We have for $j\in \Z_N$,
\begin{align*}
q(j) = &\sum_{k\in I} q_k(j) \,v_k = \sum_{k\in I} \langle q(j), v_k \rangle \,v_k\\
=&\sum_{k\in I}(\frac{1}{\sqrt{N}}\sum_{s=0}^{N-1} \widehat{q_k}(s) \,e^{2\pi i j s/N}) \,v_k 
=\frac{1}{\sqrt{N}} \sum_{s=0}^{N-1} \left(\sum_{k\in I} \widehat{q_k}(s) \,v_k\right)\,e^{2\pi i j s/N}
\end{align*}

Now, using the unicity of the decomposition we conclude that 
$$ \sum_{k\in I} \widehat{q_k}(s) \,e_k =\widehat{q}(s).$$

The transformation $\F_N$ can be put in matrix form. Let $F_N$  be the matrix $F_N =\{\frac{1}{\sqrt{N}} e^{-2\pi i k s/N} \}_{k,s=0,\dots, N-1}$.
Thus $F_N$ is an unitary matrix with $F_N^* =\{\frac{1}{\sqrt{N}} e^{2\pi i sk/N} \}_{s,k=0,\dots, N-1}$.
Now if we write, for $q$,\,$\widehat{q} \in \Hil^N$,\; ${q^T}=(q(0),\dots,q(N-1))$ and $\widehat{{q}}^T = (\widehat{q}(0),\dots,\widehat{q}(N-1))$ then 
\begin{equation} \label {FM}
\widehat{{q}}=F_N\, {q} \;\text{ and }\; F_N^*\,\widehat{{q}}= {q}.
\end{equation}


\section{The subspace $\M$} \label {SubspaceM}
In this section, we show that the space $\M$ in Definition~\ref{SpM} is a closed subspace of  $\mathcal{B}(\ell^2, \ell^\infty)$ (see Definition~\ref{Bsp}), and hence  $\M$ is also a Banach space. In addition, 
%
%
%
we prove an important property of the space $\M$ and show that it is a natural domain for the reconstruction operator $R$.

The first property  is stated  in the following Lemma.
\begin{lem} \label {CS} $\M$ is a closed subspace of $\mathcal{B}(\ell^2, \ell^\infty)$.
\end{lem}
\begin {proof}
Assume that a sequence $D^{(n)}=[d_{i,j}^{(n)}] \in \M \subset \mathcal{B}(\ell^2, \ell^\infty) $ converges to $D=[d_{i,j}]$, i.e., $\|D^{(n)}-D\|_{\ell^2\to \ell^\infty}\to 0$ as $n\to \infty$. Let $r_k(s)$ and $r^{(n)}_k(s)$ be defined as the rows $r_k(s)=\{d_{Nk+s,j}\}_j$ and $r^{(n)}_k(s)=\{d^{(n)}_{Nk+s,j}\}_j$. Since $D^{(n)}\in \M$, we have that $\lim_{k\to \infty} \|r^{(n)}_{Nk+s} - t^{(n)}_s\|_{\ell^2} = 0$.  To simplify notation, we will drop the script $s$ for the remainder of the proof. We first Show that the sequence $\{t^{(n)}\}_n$ is a Cauchy sequence. 
Since $D^{(n)}$ is convergent in $\M$, we have $\|r_k^{(m)}-r_l^{(n)}\|_{\ell^2}\le \|D^{(m)}-D^{(n)}\|_{\ell^2\to \ell^\infty}$. Thus, given $\varepsilon >0$, we can find an integer $I$ such that 
\[\|r_k^{(m)}-r_l^{(n)}\|_{\ell^2}\le \|D^{(m)}-D^{(n)}\|_{\ell^2\to \ell^\infty} < \varepsilon/3 \text{ for } \quad m,n\ge I,
\]
and get 
\begin{equation*}
\begin{split}
\|t^{(m)}-t^{(n)}\|_{\ell^2}&\le \|t^{(m)}-r_k^{(m)}\|_{\ell^2}+\|t^{(n)}-r_l^{(n)}\|_{\ell^2}+\|r_k^{(m)}-r_l^{(n)}\|_{\ell^2}\\
&\le \|t^{(m)}-r_k^{(m)}\|_{\ell^2}+\|t^{(n)}-r_l^{(n)}\|_{\ell^2}+\varepsilon/3\le \varepsilon \quad \text{ for } \quad m,n\ge I,
\end{split}
\end {equation*} 
where $k,l$ are chosen large enough to obtain the last inequality. Hence, $\{t^{(n)}\}_n$ is a Cauchy sequence. Therefore, there exists $t\in \ell^2$ such that $\lim_{n\to \infty} \|t^{(n)} - t\|_{\ell^2} = 0$. Finally, we show that $\lim_{k\to \infty} \|r_k - t\|_{\ell^2} = 0$. 
\begin{equation}\label {TI}
\|r_k - t\|_{\ell^2}\le \|r_k - r^{(n)}_k\|_{\ell^2}+\|r^{(n)}_k - t^{(n)}\|_{\ell^2}+\|t^{(n)}- t\|_{\ell^2}. 
\end{equation}
For $\varepsilon >0$ choose $I$ such that   $\|r_k-r_k^{(n)}\|_{\ell^2}\le \|D-D^{(n)}\|_{\ell^2\to \ell^\infty}< \varepsilon/3$ and $\|t^{(n)}- t\|_{\ell^2}< \varepsilon/3$ for $n\ge I$. Then, by fixing $n\ge I$, and choosing $I_2\ge I$ large enough so that  $\|r_k - r^{(n)}_k\|_{\ell^2} <\varepsilon/3$, Inequality \eqref {TI} gives
\[
\|r_k - t\|_{\ell^2}<\varepsilon \quad \text { for } k\ge I_2.
\]
\end{proof}
The second important property, stated in the lemma below, shows that the operator defined in \eqref {algR}  is  a bounded operator from $\M$ to $\Ww^N$. 
\begin {lem}\label {LimConv} 
Let $\Hil$ be a separable Hilbert space, and let $\mathcal {G}^s=\{g^s_j\}_{j\geq 1}$, $s=0,\dots,N-1$ be any $N$ Bessel sequences in $\Hil$ with optimal Bessel bound $C_{\mathcal{G}^s}$. Then the mapping
\begin{equation*}
Q:  \M   \longrightarrow   \Hil^N,
\end{equation*}

defined by
$$ Q(D) :=\left[\lim_{i \to \infty} \sum_{j=1}^{\infty} d_{Ni+s,j} g^s_j \right]_{s=1}^N $$

is a well-defined bounded operator.
\end{lem}
\begin{proof} As in the proof of he previous lemma we let $r_k(s)=\{d_{Nk+s,j}\}_j$  be the row $Nk+s$ of $D \in \M$.
Since $D\in \M$,  there exists a sequence $t_s\in \ell^2$ such that  

\begin{equation}\label{eqb21}\lim_{k\to \infty} \|r_k(s) - t_s\|_{\ell^2} = 0.\end{equation}
 Furthermore, since $\mathcal{G}^s=\{g^s_j\}_{j\geq 1}$ is a Bessel sequence in $\Hil$ with bound $ C_{\mathcal G}^s$, then 

\begin{equation}\label{eqb22}\left\|\sum\limits_{j\ge1} (r_k(s))(j)g^s_j-t_s(j)g^s_j\right\|_\Hil\le\sqrt{ C_{\mathcal G}^s}\,\|r_k(s) - t_s\|_{\ell^2}.
\end{equation}

Thus, using \eqref{eqb21} and \eqref{eqb22}, for each $s=0, \dots, N-1$, we have that 
$$\big[Q(D)\big]_s=\lim_{k\to \infty}\sum\limits_{j\ge1} (r_k(s))(j)g^s_j=\sum\limits_{j\ge1}t_s(j)g^s_j.$$ 

Moreover, since $\|r_k(s)\|_{\ell^2}\le \|D\|_{\ell^2\to \ell^\infty}$, we have $\|t_s\|_{\ell^2}\le \|D\|_{\ell^2\to \ell^\infty}$ and we obtain
$$\left\|\big[Q(D)\big]_s\right\|_\Hil\le \|t_s\|_{\ell^2}\le \sqrt{ C_{\mathcal G^s}} \|D\|_{\ell^2\to \ell^\infty} .$$
Since the inequality above holds for each $s=0,\dots,N-1$, taking $C = \max_s\{\sqrt{ C_{\mathcal G^s}}\}$ we have
$$ \|{Q}(D)\|_{\Hil^N} \leq C \|D\|_{\ell^2\to \ell^\infty}.$$
\end{proof}
 
\noindent
%
%
%
\bibliographystyle{siam}

\end{document}